\documentclass[a4paper]{amsart}

\usepackage[utf8]{inputenc}
\usepackage{amssymb,amsfonts,amsmath}
\usepackage{enumerate}
\usepackage{color}
\usepackage{mathrsfs}
\usepackage{graphicx}
\usepackage{verbatim}
\newtheorem{theorem}{Theorem}[section]
\newtheorem{corollary}[theorem]{Corollary}
\newtheorem{lemma}[theorem]{Lemma}

\newtheorem{question}[theorem]{Question}

\theoremstyle{definition}
\newtheorem{definition}[theorem]{Definition}

\newtheorem{notation}[theorem]{Notation}

\theoremstyle{remark}
\newtheorem{remark}[theorem]{Remark}

\newtheorem{example}[theorem]{Example}

\newcommand{\N}{\mathbb{N}}

\newcommand{\Z}{\mathbb{Z}}

\DeclareMathOperator{\Lab}{Lab}

\DeclareMathOperator{\vol}{vol}
\newcommand{\abs}[1]{\vert #1 \vert}
\newcommand{\norm}[1]{\left\lVert#1\right\rVert}

\newcommand{\defeq}{\mathrel{\mathop{:}}=}

\newcommand\Set[2]{\{\,#1\mid#2\,\}}

\begin{document}

\title{Bounded subgroups of relatively finitely presented groups}

\date{\today}		   	
		   			
%Geometric group theory


\author[E.~Schesler]{Eduard Schesler}
\address{Fakult\"{a}t f\"{u}r Mathematik und Informatik, FernUniversit\"{a}t in Hagen, 58084 Hagen, Germany}
\email{eduard.schesler@fernuni-hagen.de}

\thanks{The author was partially supported by the DFG grant WI 4079/4 within the SPP 2026 Geometry at infinity.}

\begin{abstract}
Given a finitely generated group $G$ that is relatively finitely presented with respect to a collection
%$H_{\Lambda} = \Set{H_{\lambda}}{\lambda \in \Lambda}$
of peripheral subgroups, we prove that every infinite subgroup $H$ of $G$ that is bounded in the relative Cayley graph of $G$ is conjugate into a peripheral subgroup.
As an application, we obtain a trichotomy for subgroups of relatively hyperbolic groups.
Moreover we prove the existence of the relative exponential growth rate for all subgroups of limit groups.
\end{abstract}

\maketitle

\section{Introduction}

The notion of a group $G$ that is hyperbolic relative to a finite set $H_{\Lambda}$ of its subgroups was introduced by Gromov~\cite{Gromov87}
%[Definition 6.8.A]
as a generalization of a word hyperbolic group.
In his definition, the groups $H \in H_{\Lambda}$ appear as stabilizers of points at infinity of a certain hyperbolic space $X$ the group $G$ acts on.
Since then, the study of relatively hyperbolic groups remained an active field of research and several characterizations of relative hyperbolicity were introduced by Bowditch~\cite{bowditch12}, Farb~\cite{Farb98}, and Osin~\cite{Osin06}.
In the latter work, Osin uses the concept of relative presentations in order to define the relative hyperbolicity of a group $G$ with respect to a set $H_{\Lambda} = \Set{H_{\lambda}}{\lambda \in \Lambda}$ of its subgroups.
To make this more precise, let $X \subseteq G$ be a symmetric subset such that $G$ is generated by $\bigcup \limits_{\lambda \in \Lambda} H_{\lambda} \cup X$.
Then we obtain a canonical epimorphism
\[
\varepsilon \colon F \defeq (\ast_{\lambda \in \Lambda} \widetilde{H}_{\lambda}) \ast F(X) \rightarrow G,
\]
where the groups $\widetilde{H}_{\lambda}$ are disjoint isomorphic copies of $H_{\lambda}$ and $F(X)$ denotes the free group over $X$.
Consider a subset $\mathcal{R} \subseteq F$ whose normal closure is the kernel of $\varepsilon$.
Then $\mathcal{R}$ gives rise to a so-called \emph{relative presentation of $G$ with respect to $H_{\Lambda}$} of the form
\begin{equation}\label{eq:rel-praes-intro-long}
\langle X, \mathcal{H}\ \vert \ S = 1, S \in \bigcup \limits_{\lambda \in \Lambda} \mathcal{S}_{\lambda},\ R = 1, R \in \mathcal{R} \rangle,
\end{equation}
where $\mathcal{H} \defeq \bigcup \limits_{\lambda \in \Lambda} (\widetilde{H}_{\lambda} \setminus \{1\})$ and $\mathcal{S}_{\lambda}$ is the set of all relations over the alphabet $\widetilde{H}_{\lambda}$.
In this framework, $G$ is said to be \emph{hyperbolic relative to $H_{\Lambda}$} if $X$ and $\mathcal{R}$ can be chosen to be finite and~\eqref{eq:rel-praes-intro-long} admits a \emph{linear relative Dehn function}.
That is, there is some $C > 0$ such that for every word $w$ of length at most $\ell$ over $X \cup \mathcal{H}$ that represents the identity in $G$, there is an equality of the form
\begin{equation}\label{eq:prod-conjugates}
w =_F \prod \limits_{i=1}^{k} f_i^{-1} R_i^{\pm 1} f_i
\end{equation}
that holds in $F$, where $k \leq C \ell$, $f_i \in F$ and $R_i \in \mathcal{R}$.
Note that in general, there is no reason to expect that for every $\ell \in \N$ and every relation $w$ of length at most $\ell$ there is a uniform upper bound $n \in \N$ such that $w$ can be written as in~\eqref{eq:prod-conjugates} with $k \leq n$.
Even if $X$ and $\mathcal{R}$ are finite, in which case we say that~\eqref{eq:rel-praes-intro-long} is a \emph{finite relative presentation for $G$}, there are easy examples where there is no such $n$, see~\cite[Example 1.3]{Osin06}.

In this paper we study groups $G$ that admit a finite relative presentation as in~\eqref{eq:rel-praes-intro-long} whose relative Dehn function $\delta_{G,H_{\Lambda}}^{rel}$ is well-defined. 
This means that for every $\ell \in \N$ there is a minimal number $\delta_{G,H_{\Lambda}}^{rel}(\ell)$, such that for every relation $w$ of length at most $\ell$ there is an expression of the form~\eqref{eq:prod-conjugates} with $k \leq \delta_{G,H_{\Lambda}}^{rel}(\ell)$.
Examples of relatively finitely presented groups that admit a well-defined, non-linear relative Dehn function were considered in~\cite{Hughes2021}.
The study of groups with a well-defined relative Dehn function, typically involves considerations in the so-called \emph{relative Cayley graph} $\Gamma(G,X \cup \mathcal{H})$ of $G$.
Since $X \cup \mathcal{H}$ can be (and usually is) infinite, it is natural to ask the following.

\begin{question}\label{que:intro}
Which subgroups of $G$ have bounded diameter in $\Gamma(G,X \cup \mathcal{H})$?
\end{question}

Note that, next to the finite subgroups of $G$, every subgroup of $G$ that can be conjugated into some of the groups $H_{\lambda}$ has bounded diameter in $\Gamma(G,X \cup \mathcal{H})$.
It turns out that for finitely generated $G$, the existence of a well-defined relative Dehn function is enough to deduce that there are no further examples of subgroups of $G$ whose diameter in $\Gamma(G,X \cup \mathcal{H})$ is finite.

\begin{theorem}\label{introthm:main}
Let $G$ be a finitely generated group.
Suppose that $G$ is relatively finitely presented with respect to a collection $H_{\Lambda} = \Set{H_{\lambda}}{\lambda \in \Lambda}$ of its subgroups and that the relative Dehn function $\delta_{G,H_{\Lambda}}^{rel}$ is well-defined.
Then every subgroup $K \leq G$ satisfies exactly one of the following conditions:
\begin{enumerate}
\item[\rm{1)}] $K$ is finite.
\item[\rm{2)}] $K$ is infinite and conjugate to a subgroup of some $H_{\lambda}$.
\item[\rm{3)}] $K$ is unbounded in $\Gamma(G, X \cup \mathcal{H})$.
\end{enumerate}
\end{theorem}

Note for example that if some of the subgroups $H_{\lambda}$ in Theorem~\ref{introthm:main} is infinite, then there is no subgroup $K \leq G$ that contains $H_{\lambda}$ as a proper subgroup of finite index.
This also follows from the fact that each $H_{\lambda}$ is almost malnormal, which is shown in~\cite[Proposition 2.36]{Osin06}.

If the group $G$ in Theorem~\ref{introthm:main} is relatively hyperbolic with respect to $H_{\Lambda}$, then it is known that a subgroup $K \leq G$ with infinite diameter in $\Gamma(G, X \cup \mathcal{H})$ contains a loxodromic element (see~\cite[Theorem~1.1]{Osin16} together with~\cite[Proposition~5.2]{Osin16}).
Recall that an element $g \in G$ is called \emph{loxodromic}
%for the left multiplication action of $G$ on $\Gamma(G,X \cup \mathcal{H})$
if the map
\[
\Z \rightarrow \Gamma(G,X \cup \mathcal{H}),\ n \mapsto g^n
\]
is a quasiisometrical embedding.
We therefore obtain the following classification of subgroups of relatively hyperbolic groups which, to the best of my knowledge, was not recorded before.

\begin{corollary}\label{introcor:main}
Let $G$ be a finitely generated group.
Suppose that $G$ is relatively hyperbolic with respect to a collection $H_{\Lambda} = \Set{H_{\lambda}}{\lambda \in \Lambda}$ of its subgroups.
Then every subgroup $K \leq G$ satisfies exactly one of the following conditions:
\begin{enumerate}
\item[\rm{1)}] $K$ is finite.
\item[\rm{2)}] $K$ is infinite and conjugate to a subgroup of some $H_{\lambda}$.
\item[\rm{3)}] $K$ contains a loxodromic element.
\end{enumerate}
\end{corollary}

As an application of Corollary~\ref{introcor:main}, we consider relative exponential growth rates in finitely generated groups.
Recall that for a finitely generated group $G$ and a finite generating set $X$ of $G$, the \emph{growth function} $\beta^X_G \colon \N \rightarrow \N$ of $G$ with respect to $X$ is defined by $\beta^X_G(n) = \abs{B^X_G(n)}$, where $B^X_G(n)$ denotes the set of all elements of $G$ that are represented by words of length at most $n$ in the generators of $X$ and $X^{-1}$.
Using Fekete's Lemma, it is easy to see that the limit $\lim \limits_{n \rightarrow \infty} \sqrt[n]{\beta^X_G(n)}$, known as the exponential growth rate of $G$ with respect to $X$, always exist (see for example~\cite{Milnor68}).
Given a subgroup $H \leq G$, a relative analogue of the exponential growth rate is obtained by counting the elements in the relative balls $B^X_H(n) \defeq B^X_G(n) \cap H$.
The resulting function
\[
\beta^X_H \colon \N \rightarrow \N,\ n \mapsto \abs{B^X_H(n)}
\]
is called the \emph{relative growth function} of $H$ with respect to $X$.
In~\cite[Remark~3.1]{Olshanskii17} Olshanskii pointed out that, unlike in the non-relative case, the limit $\lim \limits_{n \rightarrow \infty} \sqrt[n]{\beta^X_H(n)}$ does not exist in general.
As a consequence, the \emph{relative exponential growth rate of $H$ in $G$ with respect to $X$} is typically defined as $\limsup \limits_{n \rightarrow \infty} \sqrt[n]{\beta^X_H(n)}$.
Nevertheless, in many cases where the relative exponential growth rate is studied in the literature (see for example~\cite{Cohen82},~\cite{Grigorchuk80b},~\cite{Olshanskii17},~\cite{Sharp98},~\cite{CoulonDalBoSambusetti18},~\cite{DahmaniFuterWise19} where $G$ is free or hyperbolic) the limit $\lim \limits_{n \rightarrow \infty} \sqrt[n]{\beta^X_H(n)}$ is known to exists, in which case we say that the relative exponential growth rate of $H$ in $G$ exists with respect to $X$.
In the case where $G$ is a free group, the existence of the relative exponential growth rate was proven by Olshanskii in~\cite{Olshanskii17}, extending prior results of Cohen~\cite{Cohen82} and Grigorchuk~\cite{Grigorchuk80b} who have independently proven the existence for normal subgroups of $G$.
More recently, these existence results where generalized by the author to the case where $G$ is a finitely generated acylindrically hyperbolic group and $H$ is a subgroup that contains a generalized loxodromic element of $G$, see~\cite{Schesler22}.
By combining this with Corollary~\ref{introcor:main}, we will be able conclude the following.

\begin{theorem}\label{introthm:existence-subexponential}
Let $G$ be a finitely generated group that is relatively hyperbolic with respect to a collection $H_{\Lambda} = \Set{H_{\lambda}}{\lambda \in \Lambda}$ of its subgroups.
Suppose that each of the groups $H_{\lambda}$ has subexponential growth.
Then the relative exponential growth rate of every subgroup $H \leq G$ exists with respect to every finite generating set of $G$.
\end{theorem}

By Osin~\cite[Theorem 1.1]{Osin06}, each of the groups $H_{\lambda}$ in Theorem~\ref{introthm:existence-subexponential} is finitely generated so that the assumption on subexponential growth indeed makes sense.
Relatively hyperbolic groups $G$ as in Theorem~\ref{introthm:existence-subexponential} include many naturally occuring examples of groups.
A particularly interesting such class is given by limit groups, which were introduced by Zela in his solution of the Tarski problems~\cite{Sela01} and naturally generalize the class of free groups.
By work of Dahmani~\cite{Dahmani03} and Alibegovic~\cite{Alibegovic05}, Limit groups are known to be relatively hyperbolic with respect to a system of representatives for the conjugacy classes of its maximal abelian non-cyclic subgroups.
As a consequence, we obtain the following generalization of Olshanskii's existence result.

\begin{corollary}\label{introcor:existence-limit-groups}
Let $G$ be a limit group.
Then the relative exponential growth rate of every subgroup $H \leq G$ exists with respect to every finite generating set of $G$.
\end{corollary}

\textbf{\\Acknowledgments.} I would like to thank Jason Manning for a helpful conversation regarding an alternative way of proving Corollary~\ref{introcor:main}.

\section{Preliminaries}\label{sec:rel-fin-pres-grps}

In this section we introduce some definitions and properties that will be relevant for our study of relatively finitely presented groups.
More information about these groups can be found in~\cite{Osin06}.

\subsection{Relative presentations}
Let us fix a group $G$ and a collection $H_{\Lambda} = \Set{H_{\lambda}}{\lambda \in \Lambda}$ of so-called \emph{peripheral subgroups} of $G$.
Let $X \subseteq G$ be a symmetric subset such that $G$ is generated by $\bigcup \limits_{\lambda \in \Lambda} H_{\lambda} \cup X$.
Such $X$ will be referred to as a \emph{relative generating} of $G$ with respect to $H_{\Lambda}$.
Note that this gives us a canonical epimorphism
\[
\varepsilon \colon F \defeq (\ast_{\lambda \in \Lambda} \widetilde{H}_{\lambda}) \ast F(X) \rightarrow G,
\]
where the groups $\widetilde{H}_{\lambda}$ are pairwise disjoint isomorphic copies of $H_{\lambda}$ and $F(X)$ denotes the free group over $X$.
Let us also assume that $\widetilde{H}_{\lambda} \cap X = \emptyset$ for every $\lambda \in \Lambda$.
Let $N$ denote the kernel of $\varepsilon$ and let $\mathcal{R} \subseteq N$ be a subset whose normal closure in $F$ coincides with $N$.
For each $\lambda \in \Lambda$ let $\mathcal{S}_{\lambda}$ be the set of words over $\widetilde{H}_{\lambda} \setminus \{1\}$ that represents the identity in $G$.

\begin{definition}\label{def:rel-fin-praes}
With the notation above, we say that a \emph{relative presentation of $G$ with respect to $H_{\Lambda}$} is a presentation of the form
\begin{equation}\label{eq:rel-praes-intro-long-2}
\langle X, \mathcal{H}\ \vert \ S = 1, S \in \bigcup \limits_{\lambda \in \Lambda} \mathcal{S}_{\lambda},\ R = 1, R \in \mathcal{R} \rangle,
\end{equation}
where $\mathcal{H} \defeq \bigcup \limits_{\lambda \in \Lambda} (\widetilde{H}_{\lambda} \setminus \{1\})$.
The relative presentation~\eqref{eq:rel-praes-intro-long-2} is called finite if $X$ and $\mathcal{R}$ are finite.
In this case $G$ is said to be \emph{relatively finitely presented with respect to $H_{\Lambda}$}.
\end{definition}

The following result will be crucial for us.
It can be found in~\cite[Theorem~1.1]{Osin06}.

\begin{theorem}\label{thm:finiteness}
Let $G$ be a finitely generated group and let $H_{\Lambda} = \Set{H_{\lambda}}{\lambda \in \Lambda}$ be a collection of its subgroups.
Suppose that $G$ is finitely presented with respect to $H_{\Lambda}$.
Then the following conditions hold.
\begin{enumerate}
\item The collection $H_{\Lambda}$ is finite, i.e.\ $\abs{\Lambda} < \infty$.
\item Each subgroup $H_{\lambda}$ is finitely generated.
\end{enumerate}
\end{theorem}

\subsection{Relative Dehn functions}
Let $G$ be a relatively finitely presented group with a finite relative presentation as in Definition~\ref{def:rel-fin-praes}.
For each $\ell \in \N$, let $N_{\ell}$ denote the set of words of length at most $\ell$ over $X \cup \mathcal{H}$ that represent the identity in $G$.
Given $w \in N_{\ell}$, let $\vol(w) \in \N$ be minimal with the property that there is an expression of the form
\begin{equation}\label{eq:prod-conjugates-2}
w =_F \prod \limits_{i=1}^{\vol(w)} f_i^{-1} R_i^{\pm 1} f_i,
\end{equation}
where the equality is taken in $F$ and $f_i \in F$, $R_i \in \mathcal{R}$ for every $1 \leq i \leq \vol(w)$.

\begin{definition}
The \emph{relative Dehn function} for the finite relative presentation~\eqref{eq:rel-praes-intro-long-2} of $G$ is defined by
\[
\delta_{G,H_{\Lambda}}^{rel} \colon \N \rightarrow \N \cup \{\infty\},\ \ell \mapsto \sup \Set{\vol(w)}{w \in N_{\ell}}.
\]
We say that $\delta_{G,H_{\Lambda}}^{rel}$ is well-defined if $\delta_{G,H_{\Lambda}}^{rel}(\ell) < \infty$ for every $\ell \in \N$.
\end{definition}

An important class of relatively finitely presented groups with a well-defined Dehn function consists of relatively hyperbolic groups, which can be defined in terms of the relative Dehn function.

\begin{definition}\label{def:rel-hyp}
A relatively finitely presented group $G$ with a relative presentation~\eqref{eq:rel-praes-intro-long-2} is called \emph{relatively hyperbolic with respect to $H_{\Lambda}$} if there is some $C > 0$ such the relative Dehn function satisfies $\delta_{G,H_{\Lambda}}^{rel}(\ell) \leq C \ell$ for every $\ell \in \N$.
\end{definition}

Of course the relative Dehn function $\delta_{G,H_{\Lambda}}^{rel}$ depends on the finite relative presentation~\eqref{eq:rel-praes-intro-long-2}, and not just on $H_{\Lambda}$.
But as for ordinary, i.e.\ non-relative, Dehn functions of finitely presented groups, different finite relative presentations lead to asymptotically equivalent relative Dehn functions, see~\cite[Theorem 2.34]{Osin06}.
In particular, the property of $\delta_{G,H_{\Lambda}}^{rel}$ of being well-defined or bounded above by a linear function does not depend on the choice of a finite relative presentation.

\subsection{Geometry of relative Cayley graph}
Let us again consider a relatively finitely presented group $G$ with a finite relative presentation as in Definition~\ref{def:rel-fin-praes}.
The Cayley graph of $G$ with respect to $X \cup \mathcal{H}$ is called \emph{the relative Cayley graph} of $G$ and will be denoted by $\Gamma(G,X \cup \mathcal{H})$. 
In the following, will study the local geometry of $\Gamma(G,X \cup \mathcal{H})$.
In order to do so, let us fix some terminology.
Given an edge $e$ of $\Gamma(G,X \cup \mathcal{H})$, we write $\partial_{0}(e)$ to denote the initial vertex of $e$ and $\partial_{1}(e)$ to denote the terminal vertex of $e$.
A sequence $p = (e_1,\ldots,e_n)$ of edges in $\Gamma(G,X \cup \mathcal{H})$ is called a \emph{path} if $\partial_1(e_i) = \partial_0(e_{i+1})$ for $1 \leq i < n$.
If moreover $\partial_0(e_1) = \partial_1(e_n)$, then $p$ is said to be \emph{cyclic}.
The label of a path $p$ will be denoted by $\Lab(p)$.
Sometimes it is useful to forget about the initial vertex of a cyclic path $p = (e_1,\ldots,e_n)$.
To make this precise, we define the \emph{loop} associated to $p$ as the set $[p]$ of all paths of the form $(e_i,\ldots,e_n,e_1,\ldots,e_{i+1})$ for $1 \leq i \leq n$.
A \emph{subpath of a loop} $[p]$ is a subpath of some representative $p' \in [p]$.
The algebraic counterpart of a loop is the set $[w]$ of all cyclic conjugates of a word $w$ over $X \cup \mathcal{H}$, which will be referred to as a \emph{cyclic word}.
Accordingly, the label of a loop $[p]$ is defined as $\Lab([p]) \defeq [\Lab(p)]$.
Up to minor notational differences, the following definitions can be found in~\cite{Osin06}.

\begin{definition}\label{def:subwords-syllables}
Let $w$ be a word over $X \cup \mathcal{H}$.
A subword $v$ of $w$ is a \emph{$\lambda$-subword} if it consists of letters of $\widetilde{H}_{\lambda}$.
If a $\lambda$-subword $v$ of $w$ is not properly contained in any other $\lambda$-subword of $w$, then $v$ is called a \emph{$\lambda$-syllable} of $w$.
Similarly, we say that a word $v$ over $X \cup \mathcal{H}$ is a $\lambda$-subword of a cyclic word $[w]$ if it is a $\lambda$-subword of some cyclic conjugate of $w$.
If a $\lambda$-subword $v$ of $[w]$ is not properly contained in any other $\lambda$-subword of $[w]$, then $v$ is called a \emph{$\lambda$-syllable} of $[w]$.
\end{definition}

Let us now translate Definition~\ref{def:subwords-syllables} into conditions for paths in $\Gamma(G,X \cup \mathcal{H})$.

\begin{definition}\label{def:subpathes-syllables}
Let $q$ be a path in $\Gamma(G,X \cup \mathcal{H})$.
A subpath $p$ of $q$ is a \emph{$\lambda$-subpath} if $\Lab(p)$ is a $\lambda$-subword of $\Lab(q)$.
A $\lambda$-subpath $p$ of $q$ is called a \emph{$\lambda$-component} of $q$ if $\Lab(p)$ is a $\lambda$-syllable of $\Lab(q)$.
Suppose now that $q$ is cyclic, and consider the loop $[q]$ associated to $q$.
We say that a subpath $p$ of $[q]$ is a $\lambda$-subpath of $[q]$ if $\Lab(p)$ is a $\lambda$-subword of $\Lab([q])$.
If moreover $\Lab(p)$ is a $\lambda$-syllable of $\Lab([q])$, then $p$ is called a \emph{$\lambda$-component} of $[q]$.
\end{definition}

\begin{definition}\label{def:connected-components}
Let $p_1$ and $p_2$ be $\lambda$-components of a path $p$, respectively a loop $[q]$, in $\Gamma(G,X \cup \mathcal{H})$.
We say that $p_1$ and $p_2$ are \emph{connected}, if there is a path $c$ in $\Gamma(G,X \cup \mathcal{H})$ that connects a vertex of $p_1$ with a vertex of $p_2$ and $\Lab(c)$ consists of letters of $\widetilde{H}_{\lambda}$.
We say that $p_1$ is \emph{isolated} in $p$, respectively $[q]$, if there are no further $\lambda$-components of $p$, respectively $[q]$, that are connected to $p_1$.
\end{definition}

Let us now translate the notion of an isolated component of a path (loop) in a corresponding notion for syllables in (cyclic) words.

\begin{definition}
Let $w$ be a word over $X \cup \mathcal{H}$ and let $p$ be any path in $\Gamma(G,X \cup \mathcal{H})$ with $\Lab(p) = w$.
We say that two $\lambda$-syllables $v_1,v_2$ of $w$ are \emph{connected}, respectively \emph{isolated}, if the corresponding $\lambda$-components $p_1,p_2$ of $p$ are connected, respectively isolated.
If $w$ represents the identity in $G$ and $v_1,v_2$ are $\lambda$-syllables of the cyclic word $[w]$, then $v_1,v_2$ are \emph{connected}, respectively \emph{isolated}, if the corresponding $\lambda$-components $p_1,p_2$ of the loop $[p]$ are connected, respectively isolated.
\end{definition}

The following lemma is a direct consequence of Lemma~\cite[Lemma 2.27]{Osin06}.
It will help us to study the local structure of $\Gamma(G,X \cup \mathcal{H})$ and often lets us switch between the word metrics $d_X$ and $d_{X \cup \mathcal{H}}$.

\begin{lemma}\label{lem:finite-cycles}
Let $G$ be a finitely generated group with a finite generating set $X$.
Suppose that $G$ is relatively finitely presented with respect to a collection $H_{\Lambda} = \Set{H_{\lambda}}{\lambda \in \Lambda}$ of its subgroups and that the relative Dehn function $\delta_{G,H_{\Lambda}}^{rel}$ is well-defined.
Then for every $n \in \N$ there is a finite subset $\Omega_n \subseteq G$ with the property that for every cyclic path $q$ in $\Gamma(G,X \cup \mathcal{H})$ of length at most $n$ and every isolated component $p$ of the loop $[q]$, the label $\Lab(p)$ represents an element in $\Omega_n$.
\end{lemma}

\section{The alternating growth condition}

In this section we introduce the alternating growth condition, which will play a central role in our proof of Theorem~\ref{introthm:main}.

\subsection{Regular neighbourhoods}
Let us start by defining a condition for paths in graphs that can be thought of as a strong form of having no self-intersections.
%It is motivated by the notion of regular neighbourhoods in topology.

\begin{definition}\label{def:grph-reg-nbh}
Let $\Gamma$ be a graph and let $p$ be a path in $\Gamma$ that consecutively traverses the sequence $v_0,\ldots,v_n$ of vertices in $\Gamma$.
We say that $p$ has a \emph{regular neighbourhood in $\Gamma$} if every two vertices $v_i,v_j$ that can be joined by an edge in $\Gamma$ satisfy $\abs{i-j} \leq 1$.
\end{definition}

\begin{example}\label{ex:geod-implies-reg-nbh}
If $p$ is a geodesic path in a graph $\Gamma$, then $p$ has a regular neighbourhood in $\Gamma$.
\end{example}

\begin{example}\label{ex:cyclic-implies-no-reg-nbh}
If $p$ is a non-trivial cyclic path in a graph $\Gamma$, then $p$ does not have a regular neighbourhood $\Gamma$.
\end{example}

\begin{remark}\label{rem:loc-2-geod}
Note that every path $p$ that has a regular neighbourhood in a graph $\Gamma$ is locally $2$-geodesic, i.e.\ the restriction of $p$ to each subpath of length at most $2$ is geodesic.
\end{remark}

It will be useful for us to translate the concept of regular neighbourhoods to words over some generating set of a group.

\begin{definition}\label{def:word-regular-nbh}
Let $G$ be a group and let $X$ be a generating set of $G$.
A word $w$ over $X$ is called \emph{regular} (with respect to $X$) if some path $p$ in $\Gamma(G,X)$ with $\Lab(p) = w$ has a regular neighbourhood in $\Gamma(G,X)$.
\end{definition}

\begin{remark}\label{rem:char-regular-words}
Let $G$ be a group and let $X$ be a generating set of $G$.
From the definitions if directly follows that a word $w$ over $X$ is regular if and only if every subword $v$ of $w$ of length at least $2$ satisfies $\abs{v}_X \geq 2$, where $\abs{\cdot}_X$ denotes the word metric corresponding to $X$.
\end{remark}

\subsection{Sequences of alternating growth}
We want to study sequences of regular words in the context of finitely generated, relatively finitely presented groups.
Let us therefore fix a finitely generated group $G$, a finite generating set $X$ of $G$, and a collection $H_{\Lambda} = \Set{H_{\lambda}}{\lambda \in \Lambda}$ of peripheral subgroups of $G$.
Suppose that $G$ is relatively finitely presented with respect to $H_{\Lambda}$ and that the relative Dehn function $\delta_{G,H_{\Lambda}}^{rel}$ is well-defined.
As in Section~\ref{sec:rel-fin-pres-grps} we write $\widetilde{H}_{\lambda}$ to denote pairwise disjoint isomorphic copies of $H_{\lambda}$ that also intersect trivially with $X$.
Let us fix some notation in order to avoid ambiguities concerning the length and the evaluation of a word over $X \cup \mathcal{H}$, where as always $\mathcal{H} = \bigcup \limits_{\lambda \in \Lambda} (\widetilde{H}_{\lambda} \setminus \{1\})$.

\begin{notation}\label{not:word-length-vs-word-metric}
Let $w = w_1 \ldots w_{\ell}$ be a word over $X \cup \mathcal{H}$.
We write $\norm{w} = \ell$ for the word length of $w$.
The image of $w$ in $G$ will be denoted by $\overline{w}$.
For any subset $Y \subseteq G$ we write $\abs{\overline{w}}_Y$ for the length of a shortest word over $Y$ that represents $\overline{w}$.
If there is no such word, then we set $\abs{w}_Y = \infty$.
\end{notation}

\begin{definition}\label{def:ag-condition}
A sequence of words $(w^{(n)}_1 \ldots w^{(n)}_{\ell})_{n \in \N}$ of fixed length $\ell \geq 2$ over $X \cup \mathcal{H}$ satisfies the \emph{alternating growth condition} if the following conditions are satisfied:
\begin{enumerate}
\item[I)] If $w^{(n)}_{i} = x$ for some $1 \leq i \leq \ell$, $n \in \N$ and $x \in X$, then $w^{(m)}_{i} = x$ for every $m \in \N$.
In this case we say that $i$ is an \emph{index of type $X$}.
\item[II)] If $w^{(n)}_{i} \in \widetilde{H}_{\lambda}$ for some $1 \leq i \leq \ell$, $n \in \N$ and $\lambda \in \Lambda$, then $w^{(m)}_{i} \in \widetilde{H}_{\lambda}$ for every $m \in \N$.
In this case we say that $i$ is an \emph{index of type $\lambda$}.
\item[III)] The index $1$ is not of type $X$.
\item[IV)] Two consecutive indices are never of the same type.
\item[V)] If $i$ is of type $\lambda$, then $\overline{w}^{(n)}_{i} \notin H_{\mu}$ for every $\mu \in \Lambda \setminus \{ \lambda \}$ and every $n \in \N$.
If $\overline{w}^{(n)}_{i} \in \langle X \rangle$ for some $n \in \N$, then $\abs{\overline{w}^{(n)}_{i}}_X \geq n$.
\item[VI)] Each word $w^{(n)}_1 \ldots w^{(n)}_{\ell}$ is regular with respect to $X \cup \mathcal{H}$.
\end{enumerate}
\end{definition}

The following observation will be used frequently.

\begin{remark}\label{rem:components-are-edges}
Given a regular word $w$ over $X \cup \mathcal{H}$, it directly follows from the definitions that every syllable $v$ in $w$ is isolated and consists of a single edge.
\end{remark}

\subsection{Concatenating sequences of alternating growth}

In what follows we need to construct certain sequences $(w^{(n)}_1 \ldots w^{(n)}_{\ell})_{n \in \N}$ of words over $X \cup \mathcal{H}$ that satisfy the alternating growth condition such that $\ell$ can be chosen arbitrarily large.
In order to do so, we will use the following lemma which allows us to ``concatenate'' two sequences of words that satisfy the alternating growth condition such that the resulting sequence also satisfies the alternating growth condition.

\begin{lemma}\label{lem:construction-step}
With the notation above, suppose that there are two sequences $(v^{(n)}_1 \cdots v^{(n)}_M)_{n \in \N}$ and $(w^{(n)}_1 \cdots w^{(n)}_N)_{n \in \N}$ of words over $X \cup \mathcal{H}$ that satisfy the alternating growth condition.
Let $\lambda \in \Lambda$ be such that $w^{(n)}_1 \in \widetilde{H}_{\lambda}$ for some $n \in \N$.
\begin{enumerate}
\item[\rm{1)}] Suppose that $\overline{v}^{(n)}_{M} \notin H_{\lambda}$ for every $n \in \N$.
Then there is a strictly increasing sequences of natural numbers $(s_n)_{n \in \N}$ such that
%the concatenated sequence
\[
(v^{(s_n)}_{1} \cdots v^{(s_n)}_{M} w^{(s_n)}_{1} \cdots w^{(s_n)}_{N})_{n \in \N}
\]
satisfies the alternating growth condition.
\item[\rm{2)}] Suppose that $\overline{v}^{(n)}_{M} \in H_{\lambda}$ for every $n \in \N$.
Then there are strictly increasing sequences of natural numbers $(s_n)_{n \in \N},(t_n)_{n \in \N}$ such that the sequence
%of words
\[
(v^{(s_n)}_{1} \cdots v^{(s_n)}_{M-1} z^{(n)} w^{(t_n)}_{2} \cdots w^{(t_n)}_{N})_{n \in \N},
\]
where $z^{(n)} \in \widetilde{H}_{\lambda}$ is the element representing $\overline{v^{(s_n)}_{M} w^{(t_n)}_{1}} \in H_{\lambda}$, satisfies the alternating growth condition.
\end{enumerate}
\end{lemma}
\begin{proof}
Let us first prove $1)$.
Suppose that there is no such sequence $(s_i)_{i \in \N}$.
Then there are infinitely many $n \in \N$ such that $v^{(n)}_{1} \cdots v^{(n)}_{M} w^{(n)}_{1} \cdots w^{(n)}_{N}$ does not satisfy some of the conditions of Definition~\ref{def:ag-condition}.
Since I) - V) are clearly satisfied, it follows that $v^{(n)}_{1} \cdots v^{(n)}_{M} w^{(n)}_{1} \cdots w^{(n)}_{N}$ is not regular (with respect to $X \cup \mathcal{H}$) for infinitely many $n \in \N$.
By restriction to a subsequence if necessary, we can assume that none of the words $v^{(n)}_{1} \cdots v^{(n)}_{M} w^{(n)}_{1} \cdots w^{(n)}_{N}$ is regular.
Since $\overline{v}^{(n)}_{M} \notin H_{\lambda}$ for every $n \in \N$, none of the subwords $v^{(n)}_{M} w^{(n)}_1$ represents the trivial element in $G$.
Together with the assumption that $v^{(n)}_{1} \cdots v^{(n)}_{M}$ and $w^{(n)}_{1} \cdots w^{(n)}_{N}$ are regular, it follows that there is a maximal index $a_n$ such that 
%there is some minimally chosen index $b_n$ with
\begin{equation}\label{eq:cycles}
\abs{v^{(n)}_{a_n} \cdots v^{(n)}_{M} w^{(n)}_{1} \cdots w^{(n)}_{b_n}}_{X \cup \mathcal{H}} = 1
\end{equation}
for some index $b_n$.
Suppose that each $b_n$ is chosen to be minimal with respect to $a_n$.
Then there are generators $u_n \in X \cup \mathcal{H}$ such that
\[
q_n = v^{(n)}_{a_n} \cdots v^{(n)}_{M} w^{(n)}_{1} \cdots w^{(n)}_{b_n} u_n
\]
represents the identity in $G$ for every $n \in \N$.
We want to argue that $w^{(n)}_{1}$ is an isolated $\lambda$-syllable in the cyclic word $[q_n]$.
Suppose that this is not the case.
Then there are $3$ cases to consider.

\emph{Case $1:$}
$\overline{v_i^{(n)} \cdots v_M^{(n)}w_1^{(n)}} \in H_{\lambda}$ for some $a_n \leq i \leq N$.
Then $\overline{v^{(n)}_i \cdots v^{(n)}_M} \in H_{\lambda}$ and since $v^{(n)}_{1} \cdots v^{(n)}_{M}$ is regular, we obtain $i = M$.
Thus $\overline{v}^{(n)}_M \in H_{\lambda}$, in contrast to our assumption that $\overline{v}^{(n)}_M \notin H_{\lambda}$.

\emph{Case $2:$}
$\overline{w_1^{(n)} \cdots w_{i}^{(n)}} \in H_{\lambda}$ for some $2 \leq i \leq b_n$.
This is a contradiction since $w^{(n)}_{1} \cdots w^{(n)}_{N}$ is regular.

\emph{Case $3:$}
$\overline{w_1^{(n)} \cdots w_{b_n}^{(n)}u_n} \in H_{\lambda}$.
In this case we also have $\overline{v^{(n)}_{a_n} \cdots v^{(n)}_{M}} \in H_{\lambda}$.
Using again the assumption that $v^{(n)}_{1} \cdots v^{(n)}_{M}$ is regular, it follows that $a_n = M$ and $\overline{v}^{(n)}_{M} \in H_{\lambda}$, which contradicts our assumption that $\overline{v}^{(n)}_M \notin H_{\lambda}$.

Thus we see that $w^{(n)}_{1}$ is indeed an isolated $\lambda$-syllable in $[q_n]$.
Moreover we have $\norm{q_n} \leq M+N+1$ for every $n \in \N$.
From Lemma~\ref{lem:finite-cycles} it therefore follows that $\Set{\overline{w}^{(n)}_{1}}{n \in \N}$ is a finite subset of $G$.
On the other hand, the alternating growth condition ensures that $\abs{w^{(n)}_{1}}_X \geq n$ for every $n \in \N$.
This finally gives us the contradiction that arose from our assumption that there is no sequence $(s_i)_{i \in \N}$ as in the first case of the lemma.

\medskip

Let us now prove case $2)$ of the Lemma.
From the alternating growth condition we know that $\abs{w^{(n)}_{1}}_X \geq n$ for every $n \in \N$.
Thus we can choose strictly increasing sequences of natural numbers $(s_n)_{n \in \N},(t_n)_{n \in \N}$ such that $\abs{\overline{v^{(s_n)}_{M} w^{(t_n)}_{1}}}_{X} \geq n$ for every $n \in \N$.
Note that the conditions I) - V) of Definition~\ref{def:ag-condition} are clearly satisfied for 
\[
(v^{(s_n)}_{1} \cdots v^{(s_n)}_{M-1} z^{(n)} w^{(t_n)}_{2} \cdots w^{(t_n)}_{N})_{n \in \N},
\]
where $z^{(n)} \in \widetilde{H}_{\lambda}$ is the element representing $\overline{v^{(s_n)}_{M} w^{(t_n)}_{1}}$.
In order to prove the lemma it therefore suffices to show that $v^{(s_n)}_{1} \cdots v^{(s_n)}_{M-1} z^{(n)} w^{(t_n)}_{2} \cdots w^{(t_n)}_{N}$ is regular for all but finitely many $n \in \N$.
To see this, let us first consider the subwords $v^{(s_n)}_{1} \cdots v^{(s_n)}_{M-1}z^{(n)}$ and $z^{(n)}w^{(t_n)}_{2} \cdots w^{(t_n)}_{N}$.
Suppose that there is some $1 \leq i \leq M-1$ with $\abs{v^{(s_n)}_{i} \cdots v^{(s_n)}_{M-1}z^{(n)}}_{X \cup \mathcal{H}} \leq 1$.
Then there are two cases two consider.

\emph{Case $1:$} $\overline{v^{(s_n)}_{i} \cdots v^{(s_n)}_{M-1}z^{(n)}} \in H_{\lambda}$.
Then we also have $\overline{v^{(s_n)}_{i} \cdots v^{(s_n)}_{M-1}} \in H_{\lambda}$, and since $v^{(s_n)}_{1} \cdots v^{(s_n)}_{M}$ is regular, it follows that $M-1 = 1$.
But then $v^{(s_n)}_{M-1}$ and $v^{(s_n)}_{M}$ both represent elements of $H_{\lambda}$, which in turn contradicts the regularity of $v^{(s_n)}_{1} \cdots v^{(s_n)}_{M}$.

\emph{Case $2:$} $\overline{v^{(s_n)}_{i} \cdots v^{(s_n)}_{M-1}z^{(n)}} \notin H_{\lambda}$.
Then there is some $u_n \in X \cup \mathcal{H}$ that does not lie in $\widetilde{H}_{\lambda}$ such that $q_n \defeq v^{(s_n)}_{i} \cdots v^{(s_n)}_{M-1}z^{(n)}u_n$ represents the identity in $G$.
We claim that $z^{(n)}$ is an isolated syllable in the cyclic word $[q_n]$.
Indeed, otherwise there would be some $i \leq j \leq M-1$ with $\overline{v^{(s_n)}_{j} \cdots v^{(s_n)}_{M-1}z^{(n)}} \in H_{\lambda}$, which is impossible as we have seen in \emph{Case $1$}.
Moreover we have $\norm{q_n} \leq M$.
From Lemma~\ref{lem:finite-cycles} it therefore follows that $\Set{\overline{z}^{(n)}_{1}}{n \in \N}$ is a finite subset of $G$.
Since $\abs{z^{(n)}}_X \geq n$, we see that there are only finitely many $n \in \N$ such that $\abs{v^{(s_n)}_{i} \cdots v^{(s_n)}_{M-1}z^{(n)}}_{X \cup \mathcal{H}} \leq 1$ for some $1 \leq i \leq M-1$.
Thus $v^{(s_n)}_{1} \cdots v^{(s_n)}_{M-1}z^{(n)}$ is regular for all but finitely many $n \in \N$.
Symmetric argumentation shows that $z^{(n)}w^{(t_n)}_{2} \cdots w^{(t_n)}_{N}$ is regular for all but finitely many $n \in \N$.
By restriction to a subsequence if necessary, we can therefore assume that the words $v^{(s_n)}_{1} \cdots v^{(s_n)}_{M-1} z^{(n)}$ and $ w^{(t_n)}_{2} \cdots w^{(t_n)}_{N}$ are regular for every $n$.

Suppose now that $v^{(s_n)}_{1} \cdots v^{(s_n)}_{M-1} z^{(n)} w^{(t_n)}_{2} \cdots w^{(t_n)}_{N}$ is not regular.
Then we can choose $1 \leq a_n \leq M-1$ and $2 \leq b_n \leq N$ such that $v^{(s_n)}_{a_n} \cdots v^{(s_n)}_{M-1} z^{(n)} w^{(t_n)}_{2} \cdots w^{(t_n)}_{b_n}$ is a minimal subword of $v^{(s_n)}_{1} \cdots v^{(s_n)}_{M-1} z^{(n)} w^{(t_n)}_{2} \cdots w^{(t_n)}_{N}$ with
\[
\abs{v^{(s_n)}_{a_n} \cdots v^{(s_n)}_{M-1} z^{(n)} w^{(t_n)}_{2} \cdots w^{(t_n)}_{b_n}}_{X \cup \mathcal{H}} \leq 1.
\]

\emph{Case $1:$} The word $q_n \defeq v^{(s_n)}_{a_n} \cdots v^{(s_n)}_{M-1} z^{(n)} w^{(t_n)}_{2} \cdots w^{(t_n)}_{b_n}$ represents the identity in $G$.
Since $v^{(s_n)}_{1} \cdots v^{(s_n)}_{M-1} z^{(n)}$ and $z^{(n)} w^{(t_n)}_{2} \cdots w^{(t_n)}_{N}$ are regular, it follows that $z^{(n)}$ is an isolated syllable in the cyclic word $[q_n]$.
In view of Lemma~\ref{lem:finite-cycles} we see that there are only finitely many such $n$.

\emph{Case $2:$} The word $v^{(s_n)}_{a_n} \cdots v^{(s_n)}_{M-1} z^{(n)} w^{(t_n)}_{2} \cdots w^{(t_n)}_{b_n}$ does not represent an element of $H_{\lambda}$.
Then there is some $u_n \in \bigcup \limits_{\mu\in \Lambda \setminus \{\lambda\}} (\widetilde{H}_{\lambda} \setminus \{1\}) \cup X$ such that
\[
q_n \defeq v^{(s_n)}_{a_n} \cdots v^{(s_n)}_{M-1} z^{(n)} w^{(t_n)}_{2} \cdots w^{(t_n)}_{b_n}u_n
\]
represents the trivial element in $G$.
In particular, $u_n$ is not part of a $\lambda$-syllable in the cyclic word $[q_n]$.
Another application of Lemma~\ref{lem:finite-cycles} now reveals that there are only finitely many $n \in \N$ such that $v^{(s_n)}_{a_n} \cdots v^{(s_n)}_{M-1} z^{(n)} w^{(t_n)}_{2} \cdots w^{(t_n)}_{b_n}$ does not represent an element of $H_{\lambda}$.

\emph{Case $3:$} The word $v^{(s_n)}_{a_n} \cdots v^{(s_n)}_{M-1} z^{(n)} w^{(t_n)}_{2} \cdots w^{(t_n)}_{b_n}$ represents a non-trivial element in $H_{\lambda}$.
Then there is some $u_n \in \widetilde{H}_{\lambda}$ such that
\[
q_n \defeq v^{(s_n)}_{a_n} \cdots v^{(s_n)}_{M-1} z^{(n)} w^{(t_n)}_{2} \cdots w^{(t_n)}_{b_n}u_n
\]
represents the identity in $G$.
Suppose that $z^{(n)}$ is connected to some further $\lambda$-syllable in the cyclic word $[q_n]$.
Since $v^{(s_n)}_{1} \cdots v^{(s_n)}_{M-1} z^{(n)}$ and $z^{(n)} w^{(t_n)}_{2} \cdots w^{(t_n)}_{N}$ are regular, $z^{(n)}$ has to be connected to $u_n$.
Hence we obtain $\overline{z^{(n)} w^{(t_n)}_{2} \cdots w^{(t_n)}_{b_n}u_n} \in H_{\lambda}$, which implies $\overline{w^{(t_n)}_{2} \cdots w^{(t_n)}_{b_n}} \in H_{\lambda}$.
From the regularity of $z^{(n)} w^{(t_n)}_{2} \cdots w^{(t_n)}_{N}$ it therefore follows that $N = 2$.
But then $\overline{w}^{(t_n)}_{2} \in H_{\lambda}$, which contradicts the regulatity of $w^{(t_n)}_{1}w^{(t_n)}_{2} \cdots w^{(t_n)}_{N}$.
Thus $u_n$ is an isolated syllable in $[q_n]$ and a final application of Lemma~\ref{lem:finite-cycles} proves that \emph{Case $3$} can only occur finitely many times.

Altogether we have shown that $v^{(s_n)}_{1} \cdots v^{(s_n)}_{M-1} z^{(n)} w^{(t_n)}_{2} \cdots w^{(t_n)}_{N}$ is regular for all but finitely many $n \in \N$, which proves the lemma.
\end{proof}

\begin{corollary}\label{cor:induction}
Let $G$ be a finitely generated group with a finite generating set $X$.
Suppose that $G$ is relatively finitely presented with respect to a collection of peripheral subgroups $H_{\Lambda} = \Set{H_{\lambda}}{\lambda \in \Lambda}$ and that the relative Dehn function $\delta_{G,H_{\Lambda}}^{rel}$ is well-defined.
Let $(w^{(n)})_{n \in \N}$ be a sequence of words over $X \cup \mathcal{H}$ that satisfies the alternating growth condition and let $K$ be the subgroup of $G$ generated by $\Set{\overline{w}_n}{n \in \N}$.
Then there is some $C \in \N$ that satisfies the following.
For every $L \in \N$ there is a sequence of words $(v_n)_{n \in \N}$ over $X \cup \mathcal{H}$ such that:
\begin{enumerate}
\item[\rm{1)}] $(v_n)_{n \in \N}$ satisfies the alternating growth condition.
\item[\rm{2)}] The length of every word $v_n$ is bounded by $L \leq \| v_n \| \leq L + C$.
\item[\rm{3)}] Every word $v_n$ represents an element of $K$.
\end{enumerate}
\end{corollary}
\begin{proof}
Let us write $w^{(n)} = w^{(n)}_1 \cdots w^{(n)}_{\ell}$ for every $n \in \N$.
From the properties II) and III) of the alternating growth condition we know that there is some $\lambda \in \Lambda$ such that $w^{(n)}_1 \in \widetilde{H}_{\lambda}$ for every $n \in \N$.
By restriction to a subsequence if necessary, we may assume that $(w_n)_{n \in \N}$ satisfies one of the following two conditions:
\begin{enumerate}
\item[1)] $\overline{w}^{(n)}_{\ell} \notin H_{\lambda}$ for every $n \in \N$.
\item[2)] $\overline{w}^{(n)}_{\ell} \in H_{\lambda}$ for every $n \in \N$.
\end{enumerate}
Suppose that the first case is satisfied and let $k \in \N$.
Then an inductive application of the first case of Lemma~\ref{lem:construction-step} provides us with subsequences
\[
(w^{(s_{i,n})}_1 \cdots w^{(s_{i,n})}_{\ell})_{n \in \N}
\]
of $w^{(n)}$ for each $1 \leq  i \leq k$ such that the sequence of concatenated words
\[
v_n \defeq
(w^{(s_{1,n})}_1 \cdots w^{(s_{1,n})}_{\ell})
(w^{(s_{2,n})}_2 \cdots w^{(s_{2,n})}_{\ell})
\ldots
(w^{(s_{k,n})}_1 \cdots w^{(s_{k,n})}_{\ell})
\]
has length $k \ell$ and satisfies the alternating growth condition.
Thus the corollary is clearly satisfied for $C = \ell$.

Let us now consider case 2) and let $k \in \N$.
Then an inductive application of the second case of Lemma~\ref{lem:construction-step} provides us with subsequences
\[
(w^{(s_{i,n})}_1 \cdots w^{(s_{i,n})}_{\ell})_{n \in \N}
\]
of $w^{(n)}$ for each $1 \leq  i \leq k$ such that the sequence of words $v_n$ given by
\[
(w^{(s_{1,n})}_1 \cdots w^{(s_{1,n})}_{\ell-1}) z^{(t_{1,n})}
(w^{(s_{2,n})}_2 \cdots w^{(s_{2,n})}_{\ell-1}) z^{(t_{2,n})}
\ldots
z^{(t_{k-1,n})}
(w^{(s_{k,n})}_2 \cdots w^{(s_{k,n})}_{\ell}),
\]
where $z^{(t_{i,n})} \in \widetilde{H}_{\lambda}$ it the element representing $\overline{w^{(s_{i,n})}_{\ell}w^{(s_{i+1,n})}_1} \in H_{\lambda}$, satisfies the alternating growth condition.
In this case $v_n$ has length $k(\ell-1)+1$ and we see that the corollary is satisfied for $C = \ell$.\end{proof}

\section{Dichotomy of infinite subgroups}

Endowed with Corollary~\ref{cor:induction}, we are now ready to study the subgroup of a relatively finitely presented group $G$ that is generated by all the elements $\overline{w}_n$, where $(w_n)_{n \in \N}$ is a sequence that satisfies the alternating growth condition.

\begin{lemma}\label{lem:final}
Let $G$ be a finitely generated group with a finite generating set $X$.
Suppose that $G$ is relatively finitely presented with respect to a collection of peripheral subgroups $H_{\Lambda} = \Set{H_{\lambda}}{\lambda \in \Lambda}$ and that the relative Dehn function $\delta_{G,H_{\Lambda}}^{rel}$ is well-defined.
Suppose that $(w_n)_{n \in \N}$ is a sequence of words over $X \cup \mathcal{H}$ that satisfies the alternating growth condition.
Then the subgroup $K \leq G$
generated by $\Set{\overline{w}_n \in G}{n \in \N}$ is unbounded with respect to $d_{X \cup \mathcal{H}}$.
\end{lemma}
\begin{proof}
Suppose that $K$ is bounded with respect to $d_{X \cup \mathcal{H}}$, i.e.\ that there is some $N \in \N$ with $\abs{k}_{X \cup \mathcal{H}} \leq N$ for every $k \in K$.
Due to Corollary~\ref{cor:induction} there is a number $L \geq 4N$ and a sequence $(v_n)_{n \in \N}$ of words $v_n = v^{(n)}_1 \cdots v^{(n)}_{L}$ over $X \cup \mathcal{H}$ that satisfies the alternating growth condition such that each $v_n$ represents an element of $K$. 
By restriction to a subsequence, we can assume that there is some $M \in \N$ with $\abs{v_n}_{X \cup \mathcal{H}} = M \leq N$ for every $n \in \N$.
Let $u^{(n)}_1 \cdots u^{(n)}_M$ be a shortest word over $X \cup \mathcal{H}$ representing $\overline{v}_n^{-1}$.
Then each word $q_n \defeq v^{(n)}_1 \cdots v^{(n)}_{L} u^{(n)}_1 \cdots u^{(n)}_M$ represents the identity in $G$.
Recall that the alternating growth condition ensures that $v^{(n)}_1 \cdots v^{(n)}_{L}$ is regular and that two consecutive letters of $v_n$ do not lie in $X$.
It therefore follows that at least every second of its letters is an isolated syllable in $v_n$.
Thus there are at least $2N$ isolated syllables in $v_n = v^{(n)}_1 \cdots v^{(n)}_{L}$.
Note that for every $\lambda \in \Lambda$ and every $\lambda$-syllable of $u^{(n)}_1 \cdots u^{(n)}_M$, which necessarily consists of a single letter $u^{(n)}_{i}$, there is at most one $\lambda$-syllable $v^{(n)}_{j}$ in $v^{(n)}_1 \cdots v^{(n)}_{L}$ that is connected to $u^{(n)}_{i}$ in the cyclic word $[q_n]$.
Indeed, otherwise there would be a connection between two different isolated $\lambda$-syllables of $v^{(n)}_1 \cdots v^{(n)}_{L}$ by a $\lambda$-word.
This implies that there are at least $2N - M \geq N$ isolated syllables in $[q_n]$ that become arbitrarily large with respect to $X$ as $n$ goes to $\infty$.
But this is a contradiction to Lemma~\ref{lem:finite-cycles} since $\norm{q_n} \leq M+L$ for every $n \in \N$.
Thus it follows that $K$ is an unbounded subset of $\Gamma(G,X \cup \mathcal{H})$.
\end{proof}

\begin{lemma}\label{lem:infinite-intersection}
Let $G$ be a finitely generated group with a finite generating set $X$.
Suppose that $G$ is relatively finitely presented with respect to a collection of peripheral subgroups $H_{\Lambda} = \Set{H_{\lambda}}{\lambda \in \Lambda}$ and that the relative Dehn function $\delta_{G,H_{\Lambda}}^{rel}$ is well-defined.
Let $K \leq G$ be an infinite subgroup that is bounded with respect to $d_{X \cup \mathcal{H}}$.
Then there is an element $g \in G$ and an index $\eta \in \Lambda$ such that $\abs{gKg^{-1} \cap H_{\eta}} = \infty$.
\end{lemma}
\begin{proof}
Since $K$ is bounded with respect to $d_{X \cup \mathcal{H}}$, each of its conjugates $gKg^{-1}$ is a bounded subset of $\Gamma(G, X \cup \mathcal{H})$.
Let $m \in \N$ be minimal with the following property:

\bigskip

\begin{enumerate}
\item[$(\ast)$] There is a conjugate $H \defeq gKg^{-1}$ of $K$, a finite relative generating set $Y$ of $G$, and an infinite sequence $(k_n)_{n \in \N}$ of pairwise distinct elements $k_n \in H$ with $\abs{k_n}_{Y \cup \mathcal{H}} = m$ for every $n \in \N$.
\end{enumerate}

\bigskip

Let $g,Y$ and $(k_n)_{n \in \N}$ be as in $(\ast)$.
For each $n$ let $u^{(n)} = u^{(n)}_1 \cdots u^{(n)}_{m}$ be a (shortest) word over $Y \cup \mathcal{H}$ that represents $k_n$.
Due to the minimality of $m$, we can extend $Y$ to any finite relative generating set $Y'$ of $G$ such that $(\ast)$ is still satisfied for an appropriate subsequence of $(k_n)_{n \in \N}$.
Since $G$ is finitely generated, we can therefore assume that $Y$ is a symmetric generating set of $G$.

Suppose first that $m = 1$.
Then $u^{(n)}_1 \in \mathcal{H} = \bigcup \limits_{\lambda \in \Lambda} ( \widetilde{H}_{\lambda} \setminus \{ 1 \})$ for all but finitely many $n \in \N$.
Since $\Lambda$ is finite by Theorem~\ref{thm:finiteness}, there is some $\eta \in \Lambda$ such that infinitely many pairwise distinct letters $u^{(n)}_1$ lie in $\widetilde{H}_{\eta}$.
It therefore follows that $\abs{gKg^{-1} \cap H_\eta} = \infty$.
%Next we will show that the assumption $m \geq 2$ leads to a contradiction.

Let us now consider the case $m \geq 2$.
We want to modify $Y$ and $u^{(n)}$ in such a way that some subsequence of $(u^{(n)})_{n \in \N}$ satisfies the alternating growth condition.
This will be done inductively by going through the letters $u^{(n)}_i$ of $u^{(n)}$.

Suppose that $u^{(n)}_1 \in Y$ for infinitely many $n \in \N$.
Then we can choose some $x_1 \in Y$ and a subsequence $(k_{j_n})_{n \in \N}$ of $(k_n)_{n \in \N}$ with $u^{(j_n)}_1 = x_1$ for every $n$.
In this case we replace $(k_n)_{n \in \N}$ by $(k_{j_n})_{n \in \N}$.

Suppose next that $u^{(n)}_1 \in \mathcal{H}$ for all but finitely many $n \in \N$.
Since $\Lambda$ is finite, there is some $\lambda_1 \in \Lambda$ with $u^{(n)}_1 \in \widetilde{H}_{\lambda_1}$ for infinitely many $n \in \N$.
We have to consider $2$ cases.

\emph{Case $1:$} There is some $\widetilde{h}_1 \in \widetilde{H}_{\lambda_1}$ with $u^{(n)}_1 = \widetilde{h}_1$ for infinitely many $n \in \N$.
Then we restrict $(k_{n})_{n \in \N}$ to a subsequence $(k_{j_n})_{n \in \N}$ such that $u^{(j_n)}_1 = \widetilde{h}_1$ for every $n \in \N$.
Moreover we add $h_1$ and $h_1^{-1}$ to $Y$ and replace the letter $u^{(j_n)}_1 = \widetilde{h}_1 \in \widetilde{H}_{\lambda_1}$ in $u^{(j_n)}$ by $h_1 \in Y$ for every $n \in \N$.
Next we replace the resulting sequence by a subsequence that satisfies $(\ast)$, which is possible by the choice of $m$.

\emph{Case $2:$} There is no $\widetilde{h}_1 \in \widetilde{H}_{\lambda_1}$ with $u^{(n)}_1 = \widetilde{h}_1$ for infinitely many $n \in \N$.
In this case we replace $(u^{(n)})_{n \in \N}$ by a subsequence $(u^{(j_n)})_{n \in \N}$ such that $\abs{\overline{u}^{(j_n)}_1}_Y > n$ for every $n \in \N$.

We proceed analogously with the other indices $i \in \{2, \ldots ,m\}$.
The resulting sequence of words over $Y \cup \mathcal{H}$ will be denoted by $(v^{(n)}_1 \cdots v^{(n)}_{m})_{n \in \N}$.
Let $g_n \in H$ be the element represented by $v^{(n)}_1 \cdots v^{(n)}_{m}$.

Suppose that either two consecutive letters $v^{(n)}_{i},v^{(n)}_{i+1}$ or the letters $v^{(n)}_{1},v^{(n)}_{m}$ both lie in $Y$.
Then we could add $v^{(n)}_{i} v^{(n)}_{i+1}$ and $(v^{(n)}_{i} v^{(n)}_{i+1})^{-1}$ (respectively $v^{(n)}_{m} v^{(n)}_{1}$ and $(v^{(n)}_{m} v^{(n)}_{1})^{-1}$) to $Y$ in order to obtain a shorter sequence of infinitely many pairwise distinct elements of $H$ (respectively of $v^{(n)}_{-1}Hv^{(n)}_{1}$) with respect to $d_{Y \cup \mathcal{H}}$.
But this is a contradiction to the choice of $m$.
Thus it follows that neither $v^{(n)}_{i},v^{(n)}_{i+1}$ nor $v^{(n)}_{1},v^{(n)}_{m}$ both lie in $Y$.
In particular we see that we can replace $v^{(n)}_1 \cdots v^{(n)}_{m}$ by its inverse $(v^{(n)}_m)^{-1} \cdots (v^{(n)}_{1})^{-1}$ to ensure that the first letter does not lie in $Y$.
Let us therefore assume that $v^{(n)}_1$ is never contained in $Y$.
In order to prove that $(v^{(n)}_1 \cdots v^{(n)}_{m})_{n \in \N}$ satisfies the alternative growth condition, it remains to show that each $v^{(n)}_1 \cdots v^{(n)}_{m}$ is regular and that two consecutive letters $v^{(n)}_{i},v^{(n)}_{i+1}$ cannot lie in the same group $\widetilde{H}_{\lambda}$.
But these properties are direct consequences of the condition $\abs{g_n}_{Y \cup \mathcal{H}} = m$ from $(\ast)$, where $k_n$ plays the role of $g_n$.
Altogether we have shown that there is a conjugate $H$ of $K$ and a sequence $(g_n)_{n \in \N}$ of elements in $H$, that can be represented by a sequence $(v^{(n)}_1 \cdots v^{(n)}_{m})_{n \in \N}$ of words over $Y \cup \mathcal{H}$ that satisfies the alternating growth condition.
In this case Lemma~\ref{lem:final} tells us that $H$ is an unbounded subset of $\Gamma(G,Y \cup \mathcal{H})$, which clearly constradicts our assumption that $K$ is a bounded subset of $\Gamma(G,X \cup \mathcal{H})$.
Hence $m = 1$, in which case we have already proven the lemma.
\end{proof}

We are now ready to prove our main theorem.

\begin{theorem}\label{thm:main}
Let $G$ be a finitely generated group and let $X$ be a finite generating set of $G$.
Suppose that $G$ is relatively finitely presented with respect to a collection of peripheral subgroups $H_{\Lambda} = \Set{H_{\lambda}}{\lambda \in \Lambda}$ and that the relative Dehn function $\delta_{G,H_{\Lambda}}^{rel}$ is well-defined.
Then every subgroup $K \leq G$ satisfies exactly one of the following conditions:
\begin{enumerate}
\item[\rm{1)}] $K$ is finite.
\item[\rm{2)}] $K$ is infinite and conjugated to a subgroup of a peripheral subgroup.
\item[\rm{3)}] $K$ is unbounded in $\Gamma(G, X \cup \mathcal{H})$.
\end{enumerate}
\end{theorem}
\begin{proof}
Suppose that $K$ is infinite and bounded as a subset of $\Gamma(G, X \cup \mathcal{H})$.
From Lemma~\ref{lem:infinite-intersection} we know that there is an index $\eta \in \Lambda$ and an element $g \in G$ such that the $gKg^{-1} \cap H_{\eta}$ is infinite.
We can therefore choose a sequence $(h_n)_{n \in \N}$ of distinct, non-trivial elements $h_n \in gKg^{-1} \cap H_{\eta}$.
Suppose that $gKg^{-1}$ is not a subgroup of $H_{\eta}$ and let $a \in gKg^{-1} \setminus H_{\eta}$.
Let $\widetilde{h}_n \in \widetilde{H}_{\eta}$ be the element representing $h_n$.
Then, after adding $\{a,a^{-1}\}$ to $X$ if necessary, we can consider the sequence of words $(\widetilde{h}_n a)_{n \in \N}$ over $X \cup \mathcal{H}$.
We claim that $(\widetilde{h}_n a)_{n \in \N}$ has a subsequence that satisfies the alternating growth condition.
The only property that is not directly evident is that $(\widetilde{h}_n a)_{n \in \N}$ has a subsequence consisting of regular words.
Suppose that this is not the case.
Since $\Lambda$ is finite by Theorem~\ref{thm:finiteness}, it then follows that there is some $\mu \in \Lambda$ such that $\widetilde{h}_n a$ represents an element in $H_{\mu}$ for infinitely many $n \in \N$.
Then $\widetilde{h}_m a (\widetilde{h}_n a)^{-1} = \widetilde{h}_m \widetilde{h}_n^{-1}$ represents an element in $H_{\mu} \cap H_{\eta}$ whenever $\widetilde{h}_m a$ and $\widetilde{h}_n a$ both represent element of $H_{\mu}$.
Since $a$ was chosen outside of $H_{\eta}$, it moreover follows that $\widetilde{h}_n a$ can never represent an element of $H_{\eta}$.
In particular we see that $\eta \neq \mu$.
But this is a contradiction to~\cite[Proposition 2.36]{Osin06} which says that $H_{\mu} \cap H_{\eta}$ is finite for $\mu \neq \eta$.
Thus $(\widetilde{h}_n a)_{n \in \N}$ has a subsequence that satisfies the alternating growth condition.
In this case Lemma~\ref{lem:final} tells us that the subgroup $\langle \Set{ah_n}{n \in \N} \rangle$ of $gKg^{-1}$ is an unbounded in $\Gamma(G, X \cup \mathcal{H})$, which is contradicts our assumption that $K$ is bounded in $\Gamma(G, X \cup \mathcal{H})$.
Finally this proves that $gKg^{-1}$ is a subgroup of $H_{\eta}$.
\end{proof}

Let us now consider the important special case of Theorem~\ref{introthm:main}, where $G$ is relatively hyperbolic with respect to $H_{\Lambda}$.
Recall that an element $g \in G$ is called \emph{loxodromic}
%(for the left multiplication action of $G$ on $\Gamma(G,X \cup \mathcal{H})$)
if the map
\[
\Z \rightarrow \Gamma(G,X \cup \mathcal{H}),\ n \mapsto g^n
\]
is a quasiisometrical embedding.
It is known that a subgroup $K \leq G$ with infinite diameter in $\Gamma(G, X \cup \mathcal{H})$ contains a loxodromic element.
This follows from a corresponding result for acylindrially hyperbolic groups~\cite[Theorem~1.1]{Osin16} and the fact that relatively hyperbolic groups act acylindrically on the (hyperbolic) graph $\Gamma(G, X \cup \mathcal{H})$~\cite[Proposition~5.2]{Osin16}.

\begin{corollary}\label{cor:main}
Let $G$ be a finitely generated group.
Suppose that $G$ is relatively hyperbolic with respect to a collection $H_{\Lambda} = \Set{H_{\lambda}}{\lambda \in \Lambda}$ of its subgroups.
Then every subgroup $K \leq G$ satisfies exactly one of the following conditions:
\begin{enumerate}
\item[\rm{1)}] $K$ is finite.
\item[\rm{2)}] $K$ is infinite and conjugate to a subgroup of some $H_{\lambda}$.
\item[\rm{3)}] $K$ contains a loxodromic element.
\end{enumerate}
\end{corollary}

\section{Applications}

As an application of the classification of subgroups of relatively hyperbolic groups given in Corollary~\ref{cor:main}, we prove the existence of the relative exponential growth rate for all subgroups of a large variety of relatively hyperbolic groups.

\begin{definition}\label{def:rel-growth}
Let $G$ be a finitely generated group and let $X$ be a finite generating set of $G$.
Given a subgroup $H \leq G$, we define the \emph{relative growth function} of $H$ in $G$ with respect to $X$ by
\[
\beta^X_H \colon \N \rightarrow \N,\ n \mapsto \abs{B^X_H(n)},
\]
where $B^X_H(n)$ denotes the set of elements in $H$ that are represented by words of length at most $n$ over $X \cup X^{-1}$.
The relative exponential growth rate of $H$ in $G$ with respect to $X$ is defined by $\limsup \limits_{n \rightarrow \infty} \sqrt[n]{\beta^X_H(n)}$.
\end{definition}

It is natural to ask whether $\limsup$ can be replaced by $\lim$, i.e.\ whether the limit $\lim \limits_{n \rightarrow \infty} \sqrt[n]{\beta^X_H(n)}$ exists.
Unlike in the important special case $H = G$, in which this limit
%$\lim \limits_{n \rightarrow \infty} \sqrt[n]{\beta^X_G(n)}$
is well-known to exist (see e.g.~\cite{Milnor68}), it does now exist in general (see~\cite[Remark~3.1]{Olshanskii17}).
In the case where the limit $\lim \limits_{n \rightarrow \infty} \sqrt[n]{\beta^X_H(n)}$ does exist, we will say that the relative exponential growth rate of $H$ in $G$ exists with respect to $X$.
The following result provides us with a large variety of finitely generated relatively hyperbolic groups $G$ for which the relative exponential growth rate exists for every of its subgroups and generating sets.

\begin{theorem}\label{thm:existence-subexponential}
Let $G$ be a finitely generated group that is relatively hyperbolic with respect to a collection $H_{\Lambda} = \Set{H_{\lambda}}{\lambda \in \Lambda}$ of its subgroups.
Suppose that each of the groups $H_{\lambda}$ has subexponential growth.
Then the relative exponential growth rate of every subgroup $K \leq G$ exists with respect to every finite generating set of $G$.
\end{theorem}
\begin{proof}
Let $X$ be a finite generating set of $G$.
We go through the $3$ cases of Corollary~\ref{cor:main}.

Suppose first that $K$ is finite.
Then $\beta^X_K$ is eventually constant and it trivially follows that $\lim \limits_{n \rightarrow \infty} \sqrt[n]{\beta^X_K(n)}$ exists and is equal to $1$.

Let us next consider the case where $K$ contains a loxodromic element $k$.
By~\cite[Proposition~5.2]{Osin16}, $G$ acts acylindrically on the (hyperbolic) graph $\Gamma(G, X \cup \mathcal{H})$.
It this case~\cite[Theorem~1.1]{Osin16} tells us that either $G$ is virtually cyclic, in which case the claim follows trivially, or $G$ is acylindrically hyperbolic, in which case the claim is covered by~\cite[Theorem 5.8]{Schesler22}.
%Since $k$ is loxodromit it follows that $K$ has unbounded orbits in $\Gamma(G, X \cup \mathcal{H})$.

Consider now the remaining case, where $K$ is infinite and conjugated to a subgroup of some peripheral subgroup.
Thus there is some $g \in G$ and some $\lambda \in \Lambda$ such that $K \leq gH_{\lambda}g^{-1}$.
By Theorem~\ref{thm:finiteness} each $H_{\lambda}$, and hence $gH_{\lambda}g^{-1}$, is finitely generated.
We can thererore choose a finite generating set $Y$ of $gH_{\lambda}g^{-1}$.
Moreover it follows from~\cite[Lemma 5.4]{Osin06} that each peripheral subgroup, and hence $gH_{\lambda}g^{-1}$, is undistorted in $G$.
We can therefore choose a constant $C > 0$ such that
\begin{equation}\label{eq:existence-subexponential}
\beta^X_{gH_{\lambda}g^{-1}}(n) \leq \beta^Y_{gH_{\lambda}g^{-1}}(Cn)
\end{equation}
for every $n \in \N$.
By assumption each $H_{\lambda}$, and therefore $gH_{\lambda}g^{-1}$, has subexponential growth.
Thus we have $\lim \limits_{n \rightarrow \infty} \beta^Y_{gH_{\lambda}g^{-1}}(n) / a^{n} = 0$ for every $a > 1$.
In view of~\eqref{eq:existence-subexponential}, this implies that
\[
\lim \limits_{n \rightarrow \infty} \beta^X_{gH_{\lambda}g^{-1}}(n) / a^{n} = 0.
\]
Since $\beta^X_{K}(n) \leq \beta^X_{gH_{\lambda}g^{-1}}(n)$ for $n \in \N$, we see that $\lim \limits_{n \rightarrow \infty} \sqrt[n]{\beta^X_{K}(n)} = 1$ and in particular that the limit exists.
\end{proof}

%Many finitely generated relatively hyperbolic groups that appear in the literature have virtually nilpotent peripheral subgroups.
%As virtually nilpotent groups are well-known to have polynomial growth, those relatively hyperbolic groups are covered by Theorem~\ref{thm:existence-subexponential}.
%As mentioned in the introduction, limit groups are known to be relatively 
%%A particularly interesting such class of relatively hyperbolic groups are limit groups, which were introduced by Zela in his solution of the Tarski problems~\cite{Sela01}.
%In~\cite{Dahmani03} and~\cite{Alibegovic05} it is shown that a limit group is relatively hyperbolic with respect to a system of representatives for its conjugacy classes of maximal non-cyclic abelian subgroups.
%We therefore obtain the following.

%\begin{corollary}\label{cor:existence-limit-groups}
%Let $G$ be a limit group.
%Then the relative exponential growth rate of every subgroup $H \leq G$ exists with respect to every finite generating set of $G$.
%\end{corollary}

\bibliographystyle{amsplain}
\bibliography{Literaturverzeichnis}

\providecommand{\bysame}{\leavevmode\hbox to3em{\hrulefill}\thinspace}
\providecommand{\MR}{\relax\ifhmode\unskip\space\fi MR }
\providecommand{\MRhref}[2]{%
  \href{http://www.ams.org/mathscinet-getitem?mr=#1}{#2}
}
\providecommand{\href}[2]{#2}
\begin{thebibliography}{10}

\bibitem{Alibegovic05}
Emina Alibegovi\'{c}, \emph{A combination theorem for relatively hyperbolic
  groups}, Bull. London Math. Soc. \textbf{37} (2005), no.~3, 459--466.
  \MR{2131400}

\bibitem{bowditch12}
B.~H. Bowditch, \emph{Relatively hyperbolic groups}, Internat. J. Algebra
  Comput. \textbf{22} (2012), no.~3, 1250016, 66. \MR{2922380}

\bibitem{Cohen82}
Joel~M. Cohen, \emph{Cogrowth and amenability of discrete groups}, J. Funct.
  Anal. \textbf{48} (1982), no.~3, 301--309. \MR{678175}

\bibitem{CoulonDalBoSambusetti18}
R\'{e}mi Coulon, Fran\c{c}oise Dal'Bo, and Andrea Sambusetti, \emph{Growth gap
  in hyperbolic groups and amenability}, Geom. Funct. Anal. \textbf{28} (2018),
  no.~5, 1260--1320. \MR{3856793}

\bibitem{Dahmani03}
Fran\c{c}ois Dahmani, \emph{Combination of convergence groups}, Geom. Topol.
  \textbf{7} (2003), 933--963. \MR{2026551}

\bibitem{DahmaniFuterWise19}
Fran\c{c}ois Dahmani, David Futer, and Daniel~T. Wise, \emph{Growth of
  quasiconvex subgroups}, Math. Proc. Cambridge Philos. Soc. \textbf{167}
  (2019), no.~3, 505--530. \MR{4015648}

\bibitem{Farb98}
B.~Farb, \emph{Relatively hyperbolic groups}, Geom. Funct. Anal. \textbf{8}
  (1998), no.~5, 810--840. \MR{1650094}

\bibitem{Grigorchuk80b}
R.~I. Grigorchuk, \emph{Symmetrical random walks on discrete groups},
  Multicomponent random systems, Adv. Probab. Related Topics, vol.~6, Dekker,
  New York, 1980, pp.~285--325. \MR{599539}

\bibitem{Gromov87}
M.~Gromov, \emph{Hyperbolic groups}, Essays in group theory, Math. Sci. Res.
  Inst. Publ., vol.~8, Springer, New York, 1987, pp.~75--263. \MR{919829}

\bibitem{Hughes2021}
Sam Hughes, Eduardo Mart{\'\i}nez-Pedroza, and Luis Jorge~S{\'a}nchez
  Salda{\~n}a, \emph{Quasi-isometry invariance of relative filling functions},
  arXiv preprint arXiv:2107.03355 (2021).

\bibitem{Milnor68}
J.~Milnor, \emph{A note on curvature and fundamental group}, J. Differential
  Geometry \textbf{2} (1968), 1--7. \MR{0232311}

\bibitem{Olshanskii17}
A.~Yu. Olshanskii, \emph{Subnormal subgroups in free groups, their growth and
  cogrowth}, Math. Proc. Cambridge Philos. Soc. \textbf{163} (2017), no.~3,
  499--531. \MR{3708520}

\bibitem{Osin16}
D.~Osin, \emph{Acylindrically hyperbolic groups}, Trans. Amer. Math. Soc.
  \textbf{368} (2016), no.~2, 851--888. \MR{3430352}

\bibitem{Osin06}
Denis~V. Osin, \emph{Relatively hyperbolic groups: intrinsic geometry,
  algebraic properties, and algorithmic problems}, Mem. Amer. Math. Soc.
  \textbf{179} (2006), no.~843, vi+100. \MR{2182268}

\bibitem{Schesler22}
Eduard Schesler, \emph{The relative exponential growth rate of subgroups of
  acylindrically hyperbolic groups}, J. Group Theory \textbf{25} (2022), no.~2,
  293--326. \MR{4388363}

\bibitem{Sela01}
Zlil Sela, \emph{Diophantine geometry over groups i : Makanin-razborov
  diagrams}, Publications Math\'ematiques de l'IH\'ES \textbf{93} (2001),
  31--105 (en). \MR{1863735}

\bibitem{Sharp98}
Richard Sharp, \emph{Relative growth series in some hyperbolic groups}, Math.
  Ann. \textbf{312} (1998), no.~1, 125--132. \MR{1645953}

\end{thebibliography}

\end{document}